\documentclass[final]{siamltex}

\usepackage{amsbsy}
\usepackage{amsmath}
\usepackage{amssymb}
\usepackage{latexsym}
\usepackage{ifthen}
\usepackage{subfigure}
\usepackage{turnstile}
\usepackage{mathtools}
\usepackage{booktabs}
\usepackage{balance}
\usepackage{paralist}
\usepackage{cite}
\usepackage{url}
\usepackage[multiple]{footmisc}

\newtheorem{assumption}{A.}
\newtheorem{prop}{Proposition}
\newtheorem{cor}{Corollary}
\newtheorem{example}{Example}

\newcommand{\ones}{{\bf{1}}}

\newcommand{\naturals}{\mathbb{N}}
\newcommand{\prob}{\mathbb{P}}

\mathtoolsset{showonlyrefs}

\title{Robustness Properties in Fictitious-Play-Type Algorithms
\thanks{The work was partially supported by the FCT projects FCT [UID/EEA/5009/2013] and FCT [UID/EEA/50009/2013] through the Carnegie Mellon/Portugal Program managed by ICTI from FCT and by FCT Grant CMU-PT/SIA/0026/2009 and was partially supported by NSF grant CCF 1513936.
}
}

\author{Brian Swenson\footnotemark[2]\ \footnotemark[3]\, Soummya Kar\footnotemark[2]\, Jo\~{a}o Xavier\footnotemark[3]\, David S. Leslie\footnotemark[4]}

\begin{document}
\maketitle

\renewcommand{\thefootnote}{\fnsymbol{footnote}}

\footnotetext[2]{Department of Electrical and Computer Engineering,
Carnegie Mellon University, Pittsburgh, PA 15213, USA (soummyak@andrew.cmu.edu).}
\footnotetext[3]{Institute for Systems and Robotics (ISR/IST), LARSyS, Instituto Superior T\'{e}cnico, University of Lisbon (jxavier@isr.ist.utl.pt).}
\footnotetext[4]{Department of Mathematics and Statistics, Lancaster University, Lancaster, LA1 4YF, United Kingdom (d.leslie@lancaster.ac.uk)}

\renewcommand{\thefootnote}{\arabic{footnote}}

\begin{abstract}
  Fictitious play (FP) is a canonical game-theoretic learning algorithm which has been deployed extensively in decentralized control scenarios.  However standard treatments of FP, and of many other game-theoretic models, assume rather idealistic conditions which rarely hold in realistic control scenarios.  This paper considers a broad class of best response learning algorithms, that we refer to as {\em FP-type} algorithms.  In such an algorithm, given some (possibly limited) information about the history of actions, each individual forecasts the future play and chooses a (myopic) best action given their forecast. We provide a unified analysis of the behavior of FP-type algorithms under an important class of perturbations, thus demonstrating robustness to deviations from the idealistic operating conditions that have been previously assumed.  This robustness result is then used to derive convergence results for two control-relevant relaxations of standard game-theoretic applications: distributed (network-based) implementation without full observability
  and asynchronous deployment (including in continuous time).  In each case the results follow as a direct consequence of the main robustness result.
\end{abstract}



\pagestyle{myheadings}
\thispagestyle{plain}
\markboth{}{}

\section{Introduction}
Decentralized control scenarios are naturally modeled using the framework of game theory
\cite{marden2012game}. In this context, solution concepts such as Nash or correlated equilibrium can represent desirable operating conditions for the system.
A game-theoretic learning algorithm is a distributed procedure that allows a group of agents to cooperatively learn and coordinate their actions on such equilibria.

Fictitious Play (FP) \cite{Brown51} is a canonical game-theoretic learning algorithm---the FP algorithm, as well as variants thereof, have been studied in a wide range of control and optimization settings \cite{Shamma03,marden06,Lambert01,garcia2000fictitious
,lambert2003fictitious,tembine2012distributed,gass1995modified,
saad2012game,benaim2010class}.
In FP, each player tracks the empirical frequency of the actions of every other player and uses this information as a (possibly incorrect) forecast of the future behavior of game play. In particular, each player chooses their next-stage action as a myopic best response to their forecast.

FP is known to converge to Nash equilibrium (NE) in various classes of games \cite{robinson1951iterative,miyasawa1961convergence,berger2005fictitious,Mond01,Mond96,sela1999fictitious,berger2008learning},
but is known to not always do so \cite{shapley1964,jordan1993}.  Recently there have been efforts to determine the robustness of game-theoretic approaches to control, even in the absence of convergence to equilibrium \cite{ostrovski2013payoff, fudenberg1995consistency, alpcan2010network}.

However, standard treatments of FP (as well as many other learning algorithms) assume that rather idealistic conditions hold \cite{bravo2015reinforcement}. For example, in the traditional treatment of FP players are assumed to act in perfect synchrony, be capable of perfectly computing the best response in each stage, and are assumed to have instantaneous access to all information required to compute the best response.
Such assumptions are often extremely impractical---particularly, in large-scale distributed settings. This motivates the study of the robustness of learning results to perturbations occurring in practical real-world scenarios.

The paper studies the robustness of a class of \emph{FP-type} algorithms in which players are assumed to track some statistics related to the history of the game (not necessarily the empirical frequency distribution of classical FP) and form a forecast of opponent behavior using this information. As in FP, each player chooses their next-stage action as a myopic best response given their forecast.

Our main theoretical result is to show that FP-type algorithms are robust in the presence of a certain important class of perturbations. In particular, suppose that the myopic best response is perturbed so that players may sometimes choose suboptimal actions, but that the degree of suboptimality decays to zero over time.
(In the spirit of \cite{van2000weakened}, we sometimes call a FP-type process that is perturbed in this manner a \emph{weakened FP-type process}.)
We show that the fundamental learning property of a FP-type algorithm is retained in the presence of such a perturbation. In the case of classical FP, this means that convergence to NE is preserved. More generally, if a FP-type algorithm converges to some equilibrium set in the absence of perturbations, then our result can be applied to study convergence to the same equilibrium set in the presence of perturbations.

Robustness results of this kind were first studied in \cite{leslie2006generalised} for the case of classical FP. The present paper extends the approach of \cite{leslie2006generalised} to demonstrate robustness of FP-type algorithms. This greatly enhances the applicability of  game-theoretical learning theory to real-world control problems. Moreover, the results of this paper have required the development of useful new technical tools. For example, Lemma \ref{lemma_BR_epsilon_to_delta} studies $\epsilon$-best response sequences and demonstrates that such sequences may in fact be considered in terms of a more amenable sequence of so called $\delta$-perturbations (see Section \ref{sec_CTS}).
In order to demonstrate how the result can be applied to real-world control problems, we consider two example applications to control scenarios.

As a first application, we study the problem of implementing a FP-type algorithm in a distributed setting. In traditional implementations of FP-type algorithms it is assumed that players have instantaneous access to the information required to generate their forecast. However, in practical scenarios this information may often be distributed among the agents and must be disseminated using an overlaid communication graph. We present a generic method for implementing a FP-type algorithm in this setting, and we show that convergence of such an algorithm can be ensured as a consequence of the robustness result.

Distributed implementations of FP were previously studied in \cite{swenson2012ECFP}---the robustness result of this paper significantly expands the class of distributed communication protocols that can be used and extends the results to the class of FP-type algorithms. In particular, \cite{swenson2012ECFP} requires that any errors in the system decay at some minimum rate, whereas the robustness results in this paper do not require a minimum error decay rate.
In communication schemes with channel noise or random link failures, it may not be possible to achieve the error decay rates needed by \cite{swenson2012ECFP}. These important practical scenarios can, however, be handled by the methods developed in this paper.

As a second application, we consider the problem of asynchronous implementation. In many game-theoretic learning algorithms, it is assumed that players act in a perfectly synchronous manner. This assumption is unrealistic in large-scale distributed scenarios where players do not have access to a global clock. We study a practical variant of FP where players are permitted to choose actions in an asynchronous manner, and derive a mild condition under which convergence can be shown to occur.\footnote{We remark that while these applications are interesting in and of themselves, additional utility may be gained by considering them in conjunction with one another.
For example, the first application allows for \emph{synchronous} distributed implementation and the second allows for generic asynchronous implementation. Together, they allow one to study \emph{asynchronous} distributed implementation of a FP-type algorithm, using, for example, asynchronous gossip \cite{dimakis2010gossip} as a means of disseminating information amongst agents.} The proofs of these results follow as a simple consequence of the robustness result, and do not require the use of additional stochastic approximation techniques.

Applications of the robustness result are by no means limited to those presented here. For example, the companion work \cite{swenson2015CESFP} utilizes the robustness result to develop a Monte-Carlo based method that significantly mitigates computational burden of FP, and \cite{swenson2014strong} use the robustness result to develop a variant of FP that achieves convergence in \emph{strategic intentions} \cite{young2004strategic}.

The selected applications are intended to serve as a sample of the manner in which the robustness result can be applied.
Each of these applications has been studied in a variety of contexts, e.g., \cite{olshevsky2009convergence,koshal2012gossip,swenson2012ECFP,Lambert01,swenson2015CESFP,
borkar1997stochastic,perkins2012asynchronous,Fud92}. In this paper, we demonstrate how they can be treated in a unified manner and demonstrate how the robustness result can advance the state of the art in each.

The remainder of the paper is organized as follows. Section \ref{sec_prelims} sets up the notation.
Section \ref{sec_CTS} sets up the mathematical tools to be used in the proof of the main theoretical result. Section \ref{sec_FP-type} presents the notion of a FP-type algorithm, and presents our robustness result. Section \ref{sec_ecfp} presents an example FP-type process in the context of the robustness result.
Section \ref{sec_distributed_fp} studies distributed implementation
and Section \ref{sec_async_implementation} studies asynchronous implementation.
Finally, Section \ref{sec_conclusion} concludes the paper.

\section{Preliminaries}
\label{sec_prelims}
A game in normal form is represented by the tuple $\Gamma := (\mathcal{N},(Y_i,u_i)_{i\in N})$, where $\mathcal{N} = \{1,\ldots,N\}$ denotes the set of players, $Y_i$ denotes the finite set of actions available to player $i$, and $u_i:\prod_{i\in N}Y_i \rightarrow \mathbb{R}$ denotes the utility function of player $i$. Denote by $Y:= \prod_{i\in N} Y_i$ the joint action space.

For a finite set $X$, let $\Delta(X)$ denote the set of probability distributions over $X$. In particular, let $\Delta(Y_i)$ be the set of \emph{mixed} strategies available to player $i$, let $\Delta(Y_{-i})$ be the set of \emph{mixed} strategies (possibly correlated) available to all players other than $i$, and let $\Delta(Y)$ denote the set of joint mixed strategies (possibly correlated) available to all players.

In large scale distributed settings it is often convenient to study mixed strategies where players act independently.
Denote by $\Delta^N := \prod_{i\in N} \Delta(Y_i)$ the set of (independent) joint mixed strategies. That is, a strategy\footnote{As a matter of convention, we use the letters $p$ and $q$ when referring to strategies in $\Delta^N$ throughout the paper.} $p = (p_1,\ldots,p_N) \in \Delta^N$---where $p_i$ denotes the marginal strategy of player $i$---may be represented in the space $\Delta(Y)$ as the product $\prod_{i=1}^N p_i \in \Delta(Y)$. In this context we define $\Delta_{-i} := \prod_{j\not= i} \Delta_j$ to be the set of (independent) mixed strategies of players other than $i$. When convenient, we represent a mixed strategy $p\in \Delta^N$ by $p=(p_i,p_{-i})$, where $p_i \in \Delta_i$ denotes the marginal strategy of player $i$ and $p_{-i} = (p_1,\ldots,p_N)\backslash p_i \in \Delta_{-i}=\prod_{j\not= i} \Delta_j$ denotes the strategies of players other than $i$.

In the context of mixed strategies, we often wish to retain the notion of playing a single deterministic action. For this purpose, let $\ones_{y_i}$ denote the mixed strategy placing probability one on the action $y_i\in Y_i$.

For $x\in \Delta(Y)$, the expected utility  of player $i$ is given by
\begin{equation}
\label{def_mixed_U}
U_i(x) := \sum_{y \in Y} u_i(y) x(y_1,\ldots,y_n),
\end{equation}
\noindent and for $p\in \Delta^N$, the expected utility of player $i$ is given by \newline $U_i(p) := \sum_{y \in Y} u_i(y) p_1(y_1)\ldots p_N(y_N).$

Given a strategy $x_{-i} \in \Delta(Y_{-i})$, define the best response set for player $i$ by $
BR_i(x_{-i}) := \arg\max_{x_i\in\Delta(Y_i)} U_i(x_i,x_{-i}),
$
and more generally, the $\epsilon$-best-response set is given by
\begin{align} \label{def_epsilon_BR}
BR_{i,\epsilon}(x_{-i}) := \{\tilde x_{i} & \in \Delta(Y_i):~ U_i(\tilde x_i,x_{-i}) \geq \max_{x_i \in \Delta(Y_i)} U_i(x_i,x_{-i}) - \epsilon\}.
\end{align}
To keep notation simple, we sometimes employ the following abuses. The notation $y_i\in BR_{i,\epsilon}(x_{-i})$ means that $\ones_{y_i} \in BR_{i,\epsilon}(x_{-i})$. Similarly, for $y_i\in Y_i$ the notation $U_i(y_i,x_{-i})$ refers to the expected utility $U_i(\ones_{y_i},x_{-i})$.

The set of Nash equilibria is given by
\newline $NE := \{p\in \Delta^N: U_i(p_i,p_{-i}) \geq U_i( p_i',p_{-i}), ~\forall p_i' \in \Delta(Y_i),~\forall i\in \mathcal{N}\}$.

The distance between a point $x \in \mathbb{R}^m$ and a set $S \subset \mathbb{R}^m$ is given by $d(x,S) = \inf \{ \| x - x' \| : x'\in S\}$. Throughout the paper $\|\cdot\|$ denotes the $\mathcal{L}_{2}$ Euclidean norm unless otherwise specified.
We let $\mathbb{N} := \{0,1,2,\ldots\}$ denote the non-negative integers, and $\mathbb{N}_{+} := \{1,2,\ldots\}$ denote the positive integers.

Throughout, we assume the existence of probability spaces rich enough to carry out the construction of the various random variables required.  As a matter of convention, all equalities, inequalities, and set inclusions involving random quantities are interpreted almost surely (a.s.) with respect to the underlying probability measure, unless otherwise stated.

\subsection{Repeated Play} \label{sec_repeated_play}
Unless otherwise stated, the learning algorithms considered in this paper all assume the following format of repeated play \cite{young2004strategic,fudenberg1998theory}. Let a normal form game $\Gamma$ be fixed. Let players repeatedly face off in the game $\Gamma$, and for $n\in\{1,2,\ldots\}$, let $\sigma_i(n)\in \Delta(Y_i)$ denote the strategy used by player $i$ in round $n$. Let the $N$-tuple $\sigma(n) = (\sigma_1(n),\ldots,\sigma_N(n)) \in \Delta^N$ denote the joint strategy at time $n$.

\section{Difference Inclusions and Differential Inclusions}
\label{sec_CTS}
In this section we introduce the mathematical tools necessary to prove our main theoretical result.

In particular, in Section \ref{sec_FP-type} we will study the limiting behavior of a (discrete-time) FP-type process by first studying the behavior of a continuous-time analog and then relating the limit sets of the the (discrete-time) FP-type process to the limit sets of its continuous-time counterpart.

Following the approach of \cite{benaim2005stochastic}, let $F:\mathbb{R}^m\rightrightarrows \mathbb{R}^m$ denote a set-valued function mapping each point $\xi\in \mathbb{R}^m$ to a set $F(\xi) \subseteq \mathbb{R}^m$. We assume:
\begin{assumption}
\label{a_F}
(i) $F$ is a closed set-valued map.\footnote{I.e., $\mbox{Graph}(F) := \{(\xi,\eta):\eta\in F(\xi) \} $ is a closed subset of $\mathbb{R}^m \times \mathbb{R}^m$.} \\
(ii) $F(\xi)$ is a nonempty compact convex subset of $\mathbb{R}^m$ for all $\xi\in \mathbb{R}^m$.\\
(iii) For some norm $\|\cdot\|$ on $\mathbb{R}^m$, there exists $c>0$ such that for all $\xi\in \mathbb{R}^m$, $\sup_{\eta\in F(\xi)}\|\eta\| \leq c(1+\|\xi\|).$
\end{assumption}

\begin{definition}\label{def_DI_soln}
A solution for the differential inclusion $\frac{dx}{dt} \in F(x)$
with initial point $\xi \in \mathbb{R}^m$ is an absolutely continuous mapping $x:\mathbb{R}\rightarrow\mathbb{R}^m$ such that $x(0) = \xi$ and $\frac{dx(t)}{dt} \in F(x(t))$ for almost every $t\in \mathbb{R}$.
\end{definition}

In order to study the asymptotic behavior of discrete-time processes in this context, one may study the continuous-time interpolation. Formally, we define the continuous-time interpolation as follows:
\begin{definition}
Consider the discrete-time process
\begin{equation}
x(n+1) - x(n) \in \gamma(n+1) F(x(n)).
\end{equation}
Set $\tau_0 = 0 \mbox{ and } \tau_n = \sum_{i=1}^n \gamma(i)$ for $n\geq 1$ and define the continuous-time interpolated process $w:[0,\infty) \rightarrow \mathbb{R}^m$ by
\begin{equation}
w(\tau_n + s) = x(n) + s\frac{x(n+1) - x(n)}{\tau_{n+1} - \tau_{n}}, ~s\in [0,\gamma(n+1)).
\end{equation}
\end{definition}
In general, the continuous-time interpolation of a discrete-time process will not itself be a precise solution for the differential inclusion as stated in Definition \ref{def_DI_soln}. However, the interpolated process may be shown to satisfy the more relaxed solution concept---namely, that of a \emph{perturbed solution} to the differential inclusion. We first define the notion of a $\delta$-perturbation which we then use to define the notion of a perturbed solution.

\begin{definition}\label{def_delta_perturbation}
Let $F:\mathbb{R}^m \rightrightarrows \mathbb{R}^m$ be a set-valued map, and let $\delta > 0$. The $\delta$-perturbation of $F$ is given by
{\small
$$ F^{\delta}(x) := \{y\in \mathbb{R}^m: \exists z\in \mathbb{R}^m \mbox{ s.t. } \|z-x\| < \delta, ~d(y,F(z)) < \delta\}.$$
}
\end{definition}

\begin{definition}
A continuous function $y:[0,\infty)\rightarrow \mathbb{R}^m$ will be called a \emph{perturbed solution} to $F$ if it satisfies the following set of conditions:\\
(i) $y$ is absolutely continuous. \\
(ii) $\frac{dy(t)}{dt} \in F^{\delta(t)}(y(t))$ for almost every $t>0$, for some function $\delta:[0,\infty) \rightarrow \mathbb{R}$ with $\delta(t) \rightarrow 0$ as $t\rightarrow\infty$.
\end{definition}

The following proposition gives sufficient conditions under which an interpolated process will in fact be a perturbed solution.
\begin{prop}\label{prop_delta_perturbation}
Consider a discrete-time process $\{x(n)\}_{n\geq 1}$ such that
\newline $\gamma(n)^{-1}\left(x(n+1) - x(n) \right) \in F^{\delta_n}(x(n))$
where $\{\gamma(n)\}_{n\geq 1}$ is a sequence of positive numbers such that $\gamma(n) \rightarrow 0$ and $\sum_{n=1}^{\infty} \gamma(n) = \infty$,
$\{\delta_n\}_{n\geq 1}$ is a sequence of non-negative numbers converging to 0, and $\sup_n \|x(n)\| < \infty$.
Then the continuous-time interpolation of $\{x(n)\}_{n\geq 1}$ is a perturbed solution of $F$.
\end{prop}

The proof of Proposition \ref{prop_delta_perturbation} follows similar reasoning to the proof of Proposition 1.3 in \cite{benaim2005stochastic}.

Our end goal is to characterize the set of limit points of the discrete-time process $\{x(n)\}_{n\geq 1}$ by characterizing the set of limit points of its continuous-time interpolation. With that end in mind, it is useful to consider the notion of a \emph{chain-recurrent set}---a set of natural limit points for perturbed processes.

\begin{definition}
\label{def_CTS}
Let $\|\cdot\|$ be a norm on $\mathbb{R}^m$, and let $F:\mathbb{R}^m\rightrightarrows \mathbb{R}^m$ be a set valued map satisfying A. \ref{a_F}.
Consider the differential inclusion
\begin{equation}
\frac{dx}{dt} \in F(x).
\label{CTS_diff_incl}
\end{equation}

(a) Given a set $X\subset\mathbb{R}^m$ and points $\xi$ and $\eta$, we write $\xi\hookrightarrow \eta$ if for every $\epsilon >0$ and $T>0$ there exist an integer $n^*\geq 1$, solutions $x_1,\ldots,x_{n^*}$ to the differential inclusion \eqref{CTS_diff_incl}, and real numbers $t_1,\ldots,t_{n^*}$ greater than $T$ such that\\
(i) $x_i(s)\in X$, for all $0\leq s \leq t_i$ and for all $i=1,\ldots,n^*$,\\
(ii) $\|x_i(t_i) - x_{i+1}(0)\| \leq \epsilon$ for all $i=1,\ldots,n^*-1$,\\
(iii) $\|x_1(0) - \xi\| \leq \epsilon$ and $\|x_{n^*}(t_{n^*}) - \eta\|\leq \epsilon.$

(b) $X$ is said to be internally chain recurrent if $X$ is compact and $\xi \hookrightarrow \xi'$ for all $\xi,\xi'\in X$.
\end{definition}

The following theorem from \cite{benaim2005stochastic} allows one to relate the set of limit points of a perturbed solution of $F$ to the internally chain recurrent sets $F$.

\begin{theorem}[\cite{benaim2005stochastic}, Theorem 3.6]
\label{thrm_BPS_to_CTS}
Let $y$ be a bounded perturbed solution to $F$. Then the limit set of $y$, $L(y) = \bigcap_{t\geq 0} \overline{\{y(s):~s\geq t\}}$
is internally chain recurrent.
\end{theorem}

In order to eventually prove Theorem \ref{thrm_conv_to_CTS} (in the following section) we will show that the continuous-time interpolation of a FP-type process is in fact a bounded perturbed solution to the associated differential inclusion \eqref{FP_diff_inclusion}, and hence by Theorem \ref{thrm_BPS_to_CTS}, the limit points of the FP-type process are contained in the internally chain recurrent sets of the associated differential inclusion.

\section{Fictitious-Play-Type Process} \label{sec_FP-type}
In this section we will formally define the general framework of FP-type processes, and demonstrate how this encompasses several existing learning procedures.  We will then introduce the weakening of FP-type processes which allows consideration of robustness to perturbations, before proving general convergence properties of the framework.

We begin by reviewing the classical FP algorithm.

\subsection{Fictitious Play} \label{sec_FP}
Define the empirical history distribution (or empirical distribution) of player $i$ by
\begin{equation} \label{def_empirical_dist}
q_i(n) := \frac{1}{n}\sum_{s=1}^n \sigma_i(s),
\end{equation}
\noindent
where $\{\sigma_i(s)\}$ is a strategy sequence as defined in Section \ref{sec_repeated_play},
and let the joint empirical distribution profile (or just joint empirical distribution) be given by the $N$-tuple $q(n) = (q_1(n),\ldots,q_N(n))\in \Delta^N$.
A sequence of strategies $\{\sigma(n)\}_{n\geq 1}$ is said to be a \emph{fictitious play process} if for all $i\in\mathcal{N}$ and $n\geq 1$,\footnote{The initial strategy $\sigma(1)$ may be chosen arbitrarily.}
\begin{align} \label{def_FP_process}
\sigma_i(n+1) \in BR_i(q_{-i}(n)).
\end{align}
\noindent
In a FP process, it may be interpreted that players track the (marginal) empirical distribution of the actions of each opponent and treat this empirical distribution as a prediction (or forecast) of the future (mixed) strategy of that opponent. Players choose their next-stage actions as a myopic best response given this prediction.

In what follows, we will see that a \emph{FP-type} algorithm generalizes this idea---players will still form a forecast and choose their next-stage action as a myopic best response, but the manner in which the forecast can be formed will be significantly generalized.

\subsection{FP-Type Process}
\label{subsec_FP_type}\

A FP-type algorithm generalizes FP in two ways: (i) Players are permitted to track and react to a \emph{function} of the empirical history and, (ii) players consider an empirical history that may be non-uniformly weighted over time.\footnote{The class of FP-type algorithms considered here is similar to the class of best-response  algorithms considered in \cite{jordan1993}.}

In particular, let $Z$ denote a compact subset of $\mathbb{R}^m$ for some $m\in \mathbb{N}_+$ where the information that players keep track of is assumed to live. We refer to $Z$ as the \emph{observation space}. Let
$$g:\Delta(Y) \rightarrow Z$$
\noindent
be a map from the joint mixed strategy space to the observation space.
We assume the following
\begin{assumption} \label{a_g_function}
The observation map $g$ is uniformly continuous.
\end{assumption}

Let $\{z(n)\}_{n\geq 1}$ be a sequence in $Z$ that is defined recursively by letting $z(1)\in Z$ be arbitrary and for $n\geq 1$
\begin{equation}\label{def_empirical_dist_general}
z(n+1) = z(n) + \gamma(n)\left(g(\sigma(n+1)) - z(n)\right),
\end{equation}
where $\{\gamma(n)\}_{n\geq 1}$ is a predefined sequence of weights satisfying
\begin{assumption} \label{a_step_size}
$\lim_{n\rightarrow\infty} \gamma(n) = 0$, $\sum_{n\geq 1} \gamma(n) = \infty$.
\end{assumption}
We refer to $z(n)$ as the \emph{observation state} (the state $z(n)$ plays an analogous role to the empirical distribution $q(n)$ in classical FP).
In a FP-type algorithm, each player forms a prediction (or forecast) of the future behavior of opponents as a function of the observation state $z(n)$. In particular, for each player $i$, let $f_i:Z\rightarrow \Delta(Y_{-i})$ be a function mapping from the observation state to a forecast of opponents strategies. We make the following assumption
\begin{assumption} \label{a_prediction}
The forecast map $f_i$ is continuous for each $i\in\mathcal{N}$.
\end{assumption}

Given  $f = (f_1,\ldots,f_N)$, we define the best-response function $BR_{f}:Z\rightarrow \Delta(Y)$ associated with a FP-type algorithm as
$BR_{f}(z) = \prod_{i=1}^N BR_i(f_i(z)),$
where $BR_i$ is as defined in Section \ref{sec_prelims}.

When players are engaged in repeated play, we say the sequence $\{z(n)\}_{n\geq 1}$ is a FP-type process if each player's stage $(n+1)$ strategy is chosen as a myopic best response given their prediction of opponents strategies. That is, $\sigma_i(n+1) \in BR_i(f_i(z(n))),~\forall i,~\forall n$; or equivalently in recursive form (see \eqref{def_empirical_dist_general})
\begin{equation} \label{eq_empirical_dist_recursion}
z(n+1) - z(n) \in \gamma(n)\left( g(BR_{f}(z(n))) - z(n) \right).
\end{equation}
\begin{example}
Classical FP is recovered by letting $\gamma(n) = \frac{1}{n+1}$, letting the observation space be given by $Z=\Delta^N$, letting $g:\Delta(Y) \rightarrow \Delta^N$ with $g(z) = (g_1(z),\ldots,g_N(z))$, where $g_i:\Delta(Y)\rightarrow\Delta(Y_i)$ is given by $g_i(x) = \sum_{y_{-i} \in Y_{-i}} x(y_i,y_{-i})$, and for each $i$ letting $f_i:\Delta^N\rightarrow \Delta(Y_{-i})$ with $f_i(z) = (z_1,\ldots,z_{i-1},z_{i+1}\ldots,z_N)$.
\end{example}
\begin{example}
Joint Strategy FP \cite{marden06} is recovered by letting $\gamma(n) = \frac{1}{n+1}$, setting the observations space to be $Z=\Delta(Y)$, letting $g:\Delta(Y)\rightarrow\Delta(Y)$ to be the identity function and letting $f_i:\Delta(Y)\rightarrow\Delta(Y_{-i})$ be given by $f_i(z) = \sum_{y_i \in Y_i} z(y_i,y_{-i})$.
\end{example}
\begin{example}
Suppose all players use an identical action space given by $Y_i=\bar Y,~\forall i$. In this case, Empirical Centroid FP
(ECFP) \cite{swenson2012ECFP} is recovered by letting $\gamma(n) = \frac{1}{n+1}$, letting the observation space be given by $Z=\Delta(\bar Y)$, letting $g:\Delta(Y)\rightarrow \Delta(\bar Y)$ be given by $g(x) = N^{-1}\sum_{i=1}^N x_i$ where $y_i\mapsto x_i(y_i) = \sum_{y_{-i} \in Y_{-i}} x(y_i,y_{-i})$, and letting $f_i:\Delta(\bar Y) \rightarrow \Delta(Y_{-i})$ be given by $f_i(z) = (z,\ldots,z)$, i.e., the $(n-1)$-tuple containing repeated copies of $z$.\footnote{The ECFP algorithm is explored in more depth in Section \ref{sec_ecfp} in connection with the robustness result.}
\end{example}

We denote an instance of a FP-type algorithm as $\Psi = (\{\gamma(n)\}_{n\geq 1},g,(f_i)_{i=1}^n)$.

\subsection{Weakened Fictitious-Play-Type Process}\label{sec_weakened_FP}
In a FP-type algorithm it is assumed that players actions are always chosen as optimal (best response) strategies---a strong assumption.
In the spirit of \cite{van2000weakened,leslie2006generalised}, we wish to study the robustness of the convergence of a FP-type algorithm in a setting where agents may sometimes choose suboptimal actions. As we will see in later sections, this relaxation allows for a breadth of practical applications.

Formally, let the $\epsilon$-best response in this context be given by $BR_{f,\epsilon}:Z \rightarrow \Delta(Y)$, where
$BR_{f,\epsilon}(z) := \prod_{i=1}^N BR_{i,\epsilon}(f_i(z)),$
and where $BR_{i,\epsilon_n}$ is as defined in \eqref{def_epsilon_BR}.
Suppose that players choose their next-stage strategies as
\begin{equation} \label{eq_perturbed_BR}
\sigma(n+1) \in BR_{f,\epsilon_n} (z(n)),
\end{equation}
where we assume the sequence $\{\epsilon_n\}_{n\geq 1}$ satisfies
\begin{assumption} \label{a_BR_Decay}
$\lim_{n\rightarrow\infty} \epsilon_n = 0$.
\end{assumption}
We refer to the sequence $\{\epsilon_n\}_{n\geq 1}$ in \eqref{eq_perturbed_BR} as a \emph{best-response perturbation}. We refer to a sequence of strategies $\{\sigma(n)\}_{n\geq 1}$ satisfying \eqref{eq_perturbed_BR} as a \emph{weakened FP-type process} (cf. \cite{van2000weakened,leslie2006generalised}).

\subsection{Main Theoretical Result: Robustness Property for FP-Type Process}
The following theorem is the main theoretical result of the paper.  It shows that if A. \ref{a_g_function}--A. \ref{a_BR_Decay} are satisfied, then the set of limit points of a discrete-time FP-type process are contained in a chain-recurrent set of the associated differential inclusion
\begin{equation} \label{FP_diff_inclusion}
\dot{z}(t) \in g(BR_{f}(z(t)) - z(t).
\end{equation}
\begin{theorem}\label{thrm_conv_to_CTS}
Let $\Psi = (\{\gamma(n)\}_{n\geq 1},g,(f_i)_{i=1}^n)$ be a FP-type algorithm. Assume that $\Psi$ satisfies A. \ref{a_g_function}--A. \ref{a_prediction}. Assume that any best-response perturbation satisfies A. \ref{a_BR_Decay}. Then a weakened FP-type process converges to the chain recurrent set of the associated differential inclusion \eqref{FP_diff_inclusion}.
\end{theorem}

After proving the theorem,
we give an example of how the theorem can be applied to study various notions of learning in the case of a particular FP-type algorithm (see Section \ref{sec_ecfp}).

The proof of Theorem \ref{thrm_conv_to_CTS} follows directly from the following lemma together with Proposition \ref{prop_delta_perturbation} and Theorem \ref{thrm_BPS_to_CTS}. The lemma shows that for sufficiently small $\epsilon$ the $\epsilon$-best responses are contained in the $\delta$-perturbations of $BR$ for all $z$.  While this is clearly true pointwise, the uniformity in $z$ has not previously been shown. This observation was not made in \cite{leslie2006generalised} and results in a gap in the proof presented there.
\begin{lemma}
Let $\epsilon_n \rightarrow 0$ as $n\rightarrow \infty$. Then there exists a sequence $\delta_n \rightarrow 0$ such that $BR_{f,\epsilon_n}(z) \subseteq BR^{\delta_n}_{f}(z)$ uniformly for $z\in Z$.
\label{lemma_BR_epsilon_to_delta}
\end{lemma}
\begin{proof}
We work with the supremum norm on $Z$ and $\Delta(Y_i)$
throughout the proof.

Fix an arbitrary $\delta>0$. Following \cite{hurkens1995}, define the ``stability set'' of a (joint) action $y \in Y$ as
$$St(y) := \{z\in Z:~ y_i \in BR_i(f_i(z)),~\forall i\}.$$
Note that the closer that $z$ is to boundary of $St(y)$, the smaller that $\epsilon$ must be to ensure that $\epsilon$-best responses place large mass on $y$, and hence are $\delta$-perturbations of $y=BR(z)$.  To gain the uniform inclusion of the $\epsilon$-best responses in the $\delta$-perturbations we consider the interior of the sets $St(y)$ separately from neighbourhoods of boundaries of the stability sets.  To this end, extend the stability set concept to sets of actions $T\subseteq Y$ by defining
$$St(T):= \bigcap_{y\in T} St(y)$$
to be the set of $z\in Z$ such that all actions $y\in T$ are best responses to $z$.
In what follows, we will use the stability sets $St(T)$ to construct a finite cover $\{D(T)\}_{T\subseteq Y}$ of $Z$ such that $BR(f(z))\subseteq T$ for each $z\in D(T)$.  This allows us to show that $\epsilon$-best responses to elements in $D(T)$ place most of their mass on $T$, and in particular it can be shown that for each set $D(T)\subseteq Z$ there holds
\begin{equation} \label{lemma_main_eq2}
BR_{f,\epsilon}(z) \subseteq BR_{f}^{\delta}(z),~\mbox{ for all } z\in D(T)
\end{equation}
for all $\epsilon$ sufficiently small. Since the cover is finite, we can show that in fact
\begin{equation}\label{lemma_main_eq3}
BR_{f,\epsilon}(z) \subseteq BR_{f}^{\delta}(z), ~\mbox{ for all } z\in Z
\end{equation}
holds for all $\epsilon$ sufficiently small. (We note, however, that we proceed along a slightly more direct route, showing \eqref{lemma_main_eq3} without directly verifying \eqref{lemma_main_eq2}.)

To this end, note that by the upper hemicontinuity of $BR_i$ and continuity of $f_i$, we have that $St(y)$ and $St(T)$ are closed sets.
For any $\eta>0$ and any $T\subseteq Y$, let $B(St(T),\eta)$ be the open ball of radius $\eta$ about $St(T)$ which is empty if $St(T)$ is empty. Let $M = \prod_{i\in \mathcal{N}}|Y_i|$ and for each $k\in \{1,2,\ldots,M\}$ let $\mathcal{T}^k$ be the collection of all subsets $T\subseteq Y$ such that $|T| = k$. For the tuple $\eta_{>k}=(\eta_{k+1},\ldots,\eta_{M})$ define the ``exclusion set''
$$E^k(\eta_{>k}) := \bigcup_{\kappa = k+1}^{M} \bigcup_{T\in\mathcal{T}^{\kappa}} B(St(T),\eta_{\kappa})$$
to be the set of $z\in Z$ that are close to any stability sets $St(T)$ with $|T|>k$, where close is measured by the tuple $\eta_{>k}$.

We now work recursively from $k=M$ down to $k=1$. Start by letting $\eta_{M} = \delta$ and let
$$D(Y) := B(St(Y),\eta_{M}).$$
Now let $k \in \{1,\ldots,M-1\}$ and suppose $\eta_{>k}$ is given. Suppose
$T\in \mathcal{T}^{k},$ 
and let $\tilde T \subseteq Y$ such that $\tilde T\not\subseteq T.$  Then $|T\cup \tilde T| >k$, so by the definition of $E^k(\eta_{>k})$ we have that $St(T)\cap St(\tilde{T})=St(T\cup \tilde T) \subseteq E^k(\eta_{>k})$.
Therefore $St(T) \cap St(\tilde T) \cap E^k(\eta_{>k})^c = \emptyset.$
Since $E^k(\eta_{>k})$ is open by definition, the complement is closed. Therefore the sets $St(T)\cap E^k(\eta_{>k})^c$ and $St(\tilde T)$
are disjoint compact sets and either have a minimal separating distance or at least one is empty.  We can therefore fix an $\eta_k$ such that, for each $T\in\mathcal{T}^{k},$
$$D(T) := B(St(T),\eta_{k})\cap E^k(\eta_{>k})^c$$
is separated from
\begin{enumerate}
\item $St(\tilde{T})$ for all $\tilde{T}\subseteq Y$ such that $\tilde{T}\not\subseteq T$, and
\item $D(\tilde{T})$ for all $\tilde{T}\in\mathcal{T}^k$ with $\tilde{T}\neq T$.
\end{enumerate}
%
Iterating this reasoning down to $k=1$ defines the full set of $\eta_k$ values as well as $D(T)$ for all $T\subseteq Y$ with $T\neq\emptyset$.

We now show that the sets $\{D(T)\}_{T\subseteq Y}$ partition $Z$. By definition we have that $D(Y) = B(St(Y),\eta_M)$; using a backwards induction argument one may verify that
\begin{equation} \label{eq_D_B_relation}
\bigcup_{k=1}^M\bigcup_{T\in\mathcal{T}^{k}} D(T) = \bigcup_{k=1}^M\bigcup_{T\in\mathcal{T}^{k}} B(St(T),\eta_{k}).
\end{equation}
Hence,
$Z = \bigcup_{T\subseteq Y}St(T) \subseteq \bigcup_{ T\subseteq Y} D(T) \subseteq Z,$
where the equality holds because there exists a best response to any $z\in Z$, and the first containment holds by \eqref{eq_D_B_relation}.
Furthermore, by property 2) above
(and the fact that, by construction, $D(T)\cap D(\tilde T) = \emptyset$ for $|T|\not= |\tilde T|$)
we have that $D(T)\cap D(\tilde T) = \emptyset,~\forall T,\tilde T \subseteq Y,~\tilde T \not= T$.
Hence the sets $\{D(T)\}_{T\subseteq Y}$ partition $Z$.

We wish to show that for $z\in D(T)$, the $\epsilon$-best responses place most of their mass on elements in $T$.  To this end, let $T\in \mathcal{T}^k$ for arbitrary $1\leq k \leq M$, and let $\bar D(T)$ be the closure of $D(T)$. We claim that if $z\in \bar D(T)$, then all pure strategy best responses to $z$ are contained in $T$. To see this, suppose contrariwise that $z\in \bar D(T)$ has a pure strategy best response not contained in $T$. Then $z\in St(\tilde T)$ for some $\tilde T \not\subseteq T$, which violates Property 1) above.

Now define, for $z\in Z$, the set $T(z)$ to be the $T\subseteq Y$ such that $z\in D(T)$.  Also define $T_i(z) := \{y_i \in Y_i: (y_i,y_{-i}) \in T(z) \mbox{ for some } y_{-i} \in Y_{-i}\}$, so that all of Player $i$'s pure strategy best responses to $z\in Z$ are contained in $T_i(z)$.
Thus, for $z \in \bar D(T)$, for each $i$ there exists a $\xi_{i,\delta}(z)>0$ such that
\begin{equation} \label{lemma_main_eq1}
\max_{y_i \in Y_i} U_i(\ones_{y_i},f(z)) - \max_{\tilde y_i \not\in T_i(z)} U_i(\ones_{\tilde y_i},f(z)) = \xi_{i,\delta}(z).\footnote{For completeness we emphasize that $\xi_{i,\delta}$ is in fact a function of $\delta$, as well as $z$.}
\end{equation}
Since $\bar D(T)$ is compact and $U_i$ (and hence $\xi_i$) is continuous, we get $\inf\limits_{p \in \bar D(T)} \xi_{i,\delta}(z) >0$, $\forall i$. Since there are finitely many $T\subseteq Y$ and $i\in \mathcal{N}$, there exists a $\xi_{\delta}>0$ such that for each $i$ and $z \in Z$, $\max_{y_i \in Y_i} U_i(\ones_{y_i},f_i(z)) - \max_{\tilde y_i \notin T_i(z)} U_i(\ones_{\tilde y_i},f_i(z)) \geq \xi_{\delta}.$
We have shown that for any $i$ and any $z$, any action not in $T_i(z)$ receives utility less than the best response by at least an amount $\xi_\delta$.

Invoking the linearity of $z_i\mapsto U_i(z_i,z_{-i})$, it follows that for $z\in Z$, for each $i$, an $\epsilon$-best response to $z$ can put probability at most $\epsilon/\xi_{\delta}$ on actions not in $T_i(z)$. That is, for any $z\in Z$ and for any $i\in \mathcal{N}$,
$$BR_{i,\epsilon}(f_i(z)) \subseteq \Big\{x_i\in \Delta(Y_i) :~ \sum_{y_i \in T_i(z)} z_i(y_i) \geq 1-\frac{\epsilon}{\xi_{\delta}} \Big\}.$$

Let $\epsilon \leq \min\{\delta\xi_{\delta}, \delta\}$ and let $x\in BR_{f,\epsilon}(z)$. By the above, $x$ is a distance at most $\delta$ from a strategy $x'$ which places all its mass on $T(z)$. Simultaneously, by the construction of $D(T)$, $z$ is a distance at most $\delta$ from the set $St(T(z))$; i.e., there exists a $z'\in St(T(z))$ such that $d(z,z') \leq \delta$. By the definition of the stability set, we have $x' \in BR_{f}(z')$. This shows that $x \in BR_{f}^\delta(z)$. Since $z$ was arbitrary, and this holds for any $x \in BR_{f,\epsilon}(z)$ we have
$BR_{f,\epsilon}(z) \subseteq BR_{f}^{\delta}(z), ~\mbox{ for all }~ z\in Z.$
Since this holds for any $\epsilon \leq \min\{\delta\xi_{\delta}, \delta\}$, it follows that for any sequence $\epsilon_n\rightarrow 0$ there exists a sequence $\delta_n \rightarrow 0$ such that $BR_{f,\epsilon_n}(z) \subseteq BR^{\delta_n}_{f}(z)$ for any $z\in Z$.
\end{proof}

We now prove Theorem \ref{thrm_conv_to_CTS}.
\begin{proof}
By assumption, players choose their strategies according to \eqref{eq_perturbed_BR}. Applying \eqref{def_empirical_dist_general} we get the recursive form
$ \gamma(n)^{-1}(z(n+1) - z(n)) \in g(BR_{f,\epsilon_n}(z(n))) - z(n),$
where $\epsilon_n\rightarrow 0$.
By Lemma \ref{lemma_BR_epsilon_to_delta}, we know that
$\gamma(n)^{-1}(z(n+1)-z(n)) \in g(BR^{\delta_n}_f(z(n)))-z(n)$
for some sequence $\delta_n\rightarrow 0$. Let $F:Z\rightrightarrows Z$ be given by $F(z) = g(BR_{f}(z))-z$. Since $g$ is uniformly continuous, the previous equation implies that $\gamma(n)^{-1}(z(n+1)-z(n)) \in F^{\eta_n}(z(n))$ for some sequence $\eta_n\rightarrow 0$. By Proposition \ref{prop_delta_perturbation}, the continuous-time interpolation of $\{z(n)\}_{n\geq 1}$ is a bounded perturbed solution to the associated differential inclusion \eqref{FP_diff_inclusion}. The result then follows by Theorem \ref{thrm_BPS_to_CTS}.
\end{proof}

An important consequence of Theorem \ref{thrm_conv_to_CTS} is that, if one wishes to show convergence of a FP-type algorithm to some equilibrium set, one need only verify that the associated chain recurrent set is contained in the equilibrium set.

This has been shown, for example, with the set of NE and classical FP in potential games \cite{benaim2005stochastic}, two-player zero-sum games \cite{hofbauer1995stability}, and generic $2\times m$ games \cite{berger2005fictitious}. Thus, the following important result (\cite{leslie2006generalised}, Corollary 5) may also be seen as a consequence of Theorem \ref{thrm_conv_to_CTS}. As this result will arise
in the subsequent discussion, we find it convenient to state it here.
\begin{cor}[\cite{leslie2006generalised}, Corollary 5] \label{cor_FP_robustness}
Let $\Gamma$ be a potential game, two-player zero-sum game, or generic $2\times m$ game. Assume that any best-response perturbation satisfies A. \ref{a_BR_Decay}. Then the corresponding FP process converges to the set of NE in the sense that $\lim_{n\rightarrow\infty}d(q(n),NE)=0$.
\end{cor}

\section{Example: Empirical Centroid Fictitious Play} \label{sec_ecfp}
In classical FP each player $i$ is required to track the marginal empirical distribution $z_j,~j\not=i$ of every other player (see \eqref{def_FP_process}). The memory size of this vector (that must be tracked by each player) grows linearly with the number of players.  In large-scale settings it can be impractical for players to track such a large quantity of information.

In this section we consider a variant of FP in which players only track an aggregate statistic which preserves some (though not necessarily all) of the relevant information about the game action history. In the spirit of a FP-type algorithm, players form a prediction of the future behavior of opponents using the aggregate statistic.

In order to ensure the process is well defined, assume that\footnote{For the ease in exposition, ECFP is presented here in its most basic form. A more general form of ECFP is discussed in \cite{swenson2012ECFP} where this assumption may be relaxed.}
\begin{assumption} \label{a_perm_invariant}
All players use an identical action space $\bar Y$; i.e., $Y_i = \bar Y,~\forall i$. Moreover, all players use an identical permutation-invariant utility function.
\end{assumption}
More details regarding this class of games and the manner in which this assumption can be weakened can be found in \cite{swenson2012ECFP}.

In ECFP, players track and best respond to the \emph{empirical centroid distribution} $\bar q(n) \in \Delta^N$, defined as $\bar q(n) := \frac{1}{N}\sum_{i=1}^N q_i(n)$, where $q_i(n)$ is as defined in \eqref{def_empirical_dist}. In particular, each player $i$ chooses their next-stage strategy according to the rule
\begin{equation} \label{ECFP_BR_rule}
\sigma_i(n) \in BR_i(\bar q_{-i}(n-1)),
\end{equation}
where $\bar q_{-i}(n)\in \Delta(Y_{-i})$ is given by $\bar q_{-i}(n) := (\bar q(n),\ldots,\bar q(n))$, i.e., the $(n-1)$-tuple containing repeated copies of $\bar q(n)$.

Two notions of learning have been studied for ECFP. Note that both use ECFP dynamics, but achieve different learning results by using different observation spaces.
Below, we briefly review each notion in the context of the robustness result.

In order to study the first notion of learning we make the following assignments to terms from Section \ref{sec_FP-type}. Let $\gamma(n) = \frac{1}{n+1}$, let $Z=\Delta(\bar Y)$, let $g:\Delta(Y)\rightarrow \Delta(\bar Y)$ be given by $g(z) = N^{-1}\sum_{i=1}^N z_i$, where $y_i\mapsto z_i(y_i) = \sum_{y_{-i} \in Y_{-i}} z_i(y_i,y_{-i})$, and let $f_i:\Delta(\bar Y) \rightarrow \Delta(Y)$ be given by $f_i(x) = (x,\ldots,x)$, i.e., the $(n-1)$-tuple containing repeated copies of $x$.
Note that the induced dynamics comport with \eqref{ECFP_BR_rule}.


For strategies $p\in \Delta^N$, we define the set of \emph{consensus Nash equilibria} (CNE) by
$CNE := \{p\in NE:~p_1 =\ldots = p_N\}.$
Define $\overline{CNE} := \{\bar p\in \Delta(\bar Y):~ p=(\bar p,\ldots,\bar p) \in NE\}$, and note that a strategy $p\in \Delta^N$ is a CNE if and only if there exists a $\bar p \in \overline{CNE}$ such that $p=(\bar p,\ldots,\bar p)$.

It has been shown in \cite{swenson2015Robust} that the chain recurrent sets of the associated differential inclusion \eqref{FP_diff_inclusion} are contained in the $\overline{CNE}$ set. We thus obtain the following corollary to Theorem \ref{thrm_conv_to_CTS}:
\begin{cor}
Let $\Gamma$ satisfy A. \ref{a_perm_invariant}. Suppose players are engaged in a repeated play process on $\Gamma$ and choose their next stage actions according to the rule \eqref{ECFP_BR_rule}. Then players learn CNE strategies in the sense that $\lim_{n\rightarrow\infty} d( z(n),\overline{CNE})=0$, or equivalently $\lim_{n\rightarrow\infty} d(z^N(n),CNE) = 0$ where $z^N(n) = (z(n),\ldots,z(n))$ is the $N$-tuple containing repeated copies of $z(n)$.
\end{cor}

In order to study the second notion of learning we let $\gamma(n) = \frac{1}{n+1}$, let $Z = \Delta^N$, let $g:\Delta(Y)\rightarrow \Delta^N$ be given by $g(z) = (g_1(z),\ldots,g_N(z))$, where $g_i:\Delta(Y)\mapsto \Delta(\bar Y)$ with $g_i(z) = \sum_{y_{-i} \in Y_{-i}} z(y_i,y_{-i})$, and let $f_i:\Delta^N\rightarrow\Delta(Y)$ be given by $f_i(z) = \prod_{i=1}^N \bar z_i$, where $\bar z_i(y_i) = N^{-1}\sum_{j=1}^N z_j(y_i),~y_i \in \bar Y$. Note that the induced dynamics again comport with \eqref{ECFP_BR_rule}. In this case, note that the observation state lives in $\Delta^N$ and corresponds to the standard time-averaged empirical distribution familiar from classical FP.

For a strategy $p = (p_1,\ldots,p_N)\in \Delta^N$, define $\bar p := N^{-1}\sum_{i=1}^N p_i \in \Delta(\bar Y)$, and define $\bar p_{-i} := \prod_{j\not= i} \bar p \in \Delta(Y_{-i})$. Let the set of \emph{Mean-Centric Equilibria} be defined by
$MCE := \{p\in \Delta^N:~ U_i(p_i,\bar p_{-i}) \geq U_i(p_i',\bar p_{-i}),~\forall p_i' \in \Delta(\bar Y) \}$.
It has been shown in \cite{swenson2015Robust} that the chain recurrent sets of the associated differential inclusion \eqref{FP_diff_inclusion}
are contained in the set of MCE. Invoking Theorem 1 we obtain a second mode of learning as stated in the following corollary.
\begin{cor}
Let $\Gamma$ satisfy A. \ref{a_perm_invariant}. Suppose players are engaged in a repeated play process on $\Gamma$ and choose their next stage actions according to the rule \eqref{ECFP_BR_rule}. Then players learn MCE strategies in the sense that $\lim_{n\rightarrow\infty} d(z(n),MCE)=0$.
\end{cor}

\section{Application: Distributed Implementation of a FP-Type Algorithm} \label{sec_distributed_fp}
In the formulation of FP, as well as the FP-type algorithm, it is implicitly assumed that each agent has instantaneous access to all information required to compute her next-stage action. For example, in classical FP (Section \ref{sec_FP}) each agent is assumed to have perfect knowledge of the empirical distribution $q(n)$ (see \eqref{def_empirical_dist}) in order to choose an action in stage $n+1$.
This assumption can be impractical in large-scale settings where physical limitations may hinder agents' ability to directly communicate with one another.

One approach to mitigate this problem is to assume that agents are equipped with an overlaid communication graph through which information may be gradually disseminated through the course of the learning process \cite{swenson2012ECFP,koshal2012gossip,eksin2015distributed}. In particular, suppose the following assumption holds:
\begin{assumption}
Agents may observe only their own actions. However, agents are equipped with a (possibly sparse) interagent communication graph $G=(\mathcal{V},\mathcal{E})$. Agents may exchange information with neighboring agents (as defined by the graph $G$) once per iteration of the repeated play.
\end{assumption}

Within this framework, agents engaged in a FP-type process may not have perfect knowledge of the observation state $z(n)$. Instead, let $\hat z^i(n)$ be an estimate that agent $i$ maintains of $z(n)$.

A prototypical distributed implementation of a FP-type algorithm is given below.

\subsection{Distributed FP-Type Algorithm}
$~$\\
\noindent \textit{Initialize}\\
\noindent (i) Initialize the state estimate $\hat z^i(1)$.\footnote{The initialization of $\hat z^i(n)$ may be subject to some conditions depending on the particular information dissemination scheme used \cite{swenson2012ECFP,dimakis2010gossip}. See discussion below for more details.} Let players choose an arbitrary initial action.\\

\noindent \textit{Iterate} $(n\geq 1)$\\
\noindent (ii) Each agent $i$ chooses a next-stage strategy according to the rule $\sigma_i(n+1) \in BR_i(f_i(\hat z^i(n)),$
where $f_i(\cdot)$ satisfies A. \ref{a_prediction}.
The (true) observation state at time $(n+1)$ is given by $z(n+1) = z(n) + \gamma(n)\left( g(\sigma_i(n+1)) - z(n) \right)$. (It is not assumed that players have knowledge of $(z(n))$.)\\

\noindent (iii) Each agent $i$ may engage in one round of information exchange with neighboring agents (as defined by $G$) and update their estimate $\hat z^i(n+1)$ using the information obtained.

\subsection{Discussion}
Analysis of the the above algorithm prototype reveals that step (ii) may be seen as a best response perturbation (this follows from the Lipschitz continuity of $U_i$). It is straightforward to show that if $\|\hat z^i(n) - z(n)\|\rightarrow 0 ,~\forall i$, as $n\rightarrow\infty$ then A. \ref{a_BR_Decay} holds, and hence the process falls under the purview of Theorem \ref{thrm_conv_to_CTS}.

This has been applied, for example, in order to develop distributed implementations of FP and ECFP \cite{swenson2012ECFP} where the update of the empirical distribution estimate in step (iii) is carried out using a type of (synchronous) consensus recursion \cite{dimakis2010gossip}. We note, however, that the convergence results for the distributed algorithms in \cite{swenson2012ECFP} relies on an alternative form of the robustness property which required strong assumptions. In particular, it was required that error in players estimates decay as $\|\hat z^i(n) - z(n)\|=O(\frac{\log{t}}{t^r}),~r>0$.


The robustness result in this paper relies on the significantly weaker assumption that $\|\hat z^i(n) - z(n)\| \rightarrow 0$ (cf. A. \ref{a_BR_Decay}); in particular, the rate at which this goes to zero does not matter.

The protocol used to form the estimate $\hat z^i(n)$ in step (iv) is intentionally crafted to be broad in order to emphasize that a wide variety of information dissemination protocols may be used. Using the more powerful robustness result of this paper one may extend the approach of \cite{swenson2012ECFP}, demonstrating convergence of distributed implementations of FP-type algorithms in settings where players use more realistic communication protocols---e.g., asynchronous gossip \cite{dimakis2010gossip} (cf., Section \ref{sec_async_implementation}), a communication framework in which the communication graph suffers from random link dropouts \cite{kar2009distributed}, or otherwise changing topology \cite{ren2005consensus}.

\section{Application: Asynchronous Implementation of Fictitious Play}\label{sec_async_implementation}
The classical FP algorithm \eqref{def_FP_process} implicitly assumes a form of global synchronization. In particular, note that each agent must choose their stage $n$ action before any other agent chooses their stage $(n+1)$ action. In practice, such synchronization is often infeasible in large-scale distributed systems.

In this section we use the robustness result to study a variant of FP in which agents are permitted to act in an asynchronous manner. While asynchronous learning schemes would usually be analysed using asynchronous stochastic approximation (e.g. \cite{perkins2012asynchronous}) we show in this section that asynchronicity can be handled in a more straightforward manner by simply using our robustness results.
In particular, using Theorem \ref{thrm_conv_to_CTS} we develop a mild sufficient condition under which an ``asynchronous FP process'' can be shown to converge to the set of NE in the same sense as classical FP.

The initial model of asynchronous FP that we study in Section \ref{sec_async_FP} is somewhat abstract---it is this feature that allows us to capture a broad range of asynchronous processes. After introducing this model and proving convergence results (Section \ref{sec_async_FP}), we then provide simple examples of highly practical real world models that readily fall within this framework (Sections \ref{sec_CT_implementation}-\ref{sec_adaptive_clocks}).

We begin by introducing the notion of asynchronous repeated play learning---a slight modification of classical repeated play introduced in Section \ref{sec_repeated_play}.

\subsection{Asynchronous Repeated Play Learning}\label{sec_async_RP}
In order to model asynchrony, we consider an extension of the classical repeated play framework of Section \ref{sec_repeated_play} in which players may be ``active'' in some rounds and ``idle'' in others.

Let $n\in \naturals$, and let $\{X_i(n)\}_{n\geq 1}$, be a sequence of (deterministic or random) variables $X_i(n) \in \{0,1\}$ indicating the rounds in which player $i$ is active. Let $N_i(n)$ count the number of rounds in which player $i$ has been active up to and including time $n$; i.e., $N_i(n) := \sum_{s=1}^n X_i(s).$
Let $\sigma_i(n)$ represent the strategy chosen by player $i$ in round $n$. Let the empirical distribution of player $i$ be defined in this setting as $q_i(n) := \frac{1}{N_i(n)}\sum_{s=1}^n \sigma_i(s)X_i(s).$

\subsection{Fictitious Play with Asynchronous Updates} \label{sec_async_FP}
Within the generalized repeated-play framework given above, we say a sequence of strategies $\{\sigma(n)\}_{n\geq 1}$ is a \emph{FP process with asynchronous updates} (or asynchronous FP process) if for $n\geq 1$,\footnote{Let $X_i(1) = 1, ~\forall i$ and let the initial action $\sigma_i(1)$ be chosen arbitrarily for all $i$. Moreover, for convenience in notation we have used an inclusion in \eqref{eq_async_FP_proccess}. However, if $X_i(n+1)\not= 1$, then the inclusion should be interpreted as an equality: $\sigma_i(n+1) = \sigma_i(n)$.}
\begin{equation} \label{eq_async_FP_proccess}
\sigma_i(n+1) \in
\begin{cases}
BR_i(q_{-i}(n))
& \mbox{ if } X_i(n+1) = 1,\\
\sigma_i(n) & \mbox{ otherwise}.
\end{cases}
\end{equation}
This models a scenario in which each player $i$ may update her action in round $(n+1)$ according to traditional best-response dynamics only if $X_i(n+1)=1$; otherwise, the action of player $i$ persists from the previous round.
\footnote{Note that classical FP of Section \ref{sec_FP} may be seen as a special case within this framework with $X_i(n) = 1,~\forall i,n$.}

As a consequence of Corollary \ref{cor_FP_robustness}, the following assumption is sufficient (to be shown) to ensure that the FP process defined in \eqref{eq_async_FP_proccess} leads to NE learning in potential games:
\begin{assumption}
\label{a_synchrony}
(i) For each $i$ there holds $\lim_{n\rightarrow\infty} N_i(n)=\infty$; (ii) for all $i,j$ there holds, $\lim_{n\rightarrow\infty} \frac{N_i(n)}{N_j(n)} = 1$.
\end{assumption}
Part (i) in the above assumption ensures that players are active in infinitely many rounds. Part (ii) ensures that the number of actions taken by each player remain relatively close; effectively (ii) ensures that players obtain a weak form of synchronization.

The following theorem is the main theoretical result of this Section. It shows that under the above assumption, FP with asynchronous updates achieves NE learning. It will be shown to follow as a consequence of the robustness result.
\begin{theorem}\label{thrm_async_FP}
Let $\Gamma$ be a potential game. Let the action sequence $\{\sigma(n)\}_{n\geq 1}$ be determined according to a FP process with asynchronous updates and assume A. \ref{a_synchrony} holds. Then players learn NE strategies in the sense that $\lim_{n\rightarrow \infty} d(q(n),NE) = 0$.
\end{theorem}



In order to prove Theorem \ref{thrm_async_FP} we will study an underlying (synchronous) FP process that is embedded in the asynchronous FP process defined in \eqref{eq_async_FP_proccess}. We begin by presenting some additional definitions that allow us to study the embedded process.

In particular, for $s\in \mathbb{N}_{+}$ define the following terms:

$\tau_i(s) := \sup\{n\in\mathbb{N}_+:~ N_i(n)\leq s \},$
$\tilde \sigma_i(s) := \sigma_i(\tau_i(s)),$
$\tilde \sigma(s) := (\tilde \sigma_1(s),\ldots,\tilde \sigma_{N}(s)),$
$\tilde q_i(s) := q_i(\tau_i(s)),$
$\tilde q(s) := (\tilde q_1(s),\ldots,\tilde q_{N}(s)),$
$\hat  q^i_j(s) := q_j(\tau_i(s+1)-1),$
$\hat q^i(s) := (\hat q^i_1(s),\ldots, \hat q^i_{N}(s))$.

In words, the term $\tau_i(s)$ denotes the round number when player $i$ is active for the $s$-th time.
The terms marked with a $\sim$ correspond to the embedded (synchronous) FP process that we will study in the proof of Theorem \ref{thrm_async_FP}.

When studying the embedded (synchronous) FP process $\{\tilde \sigma(s)\}_{s\geq 1}$, it will be important to characterize the terms to which players are best responding. With this in mind, note that per \eqref{eq_async_FP_proccess}, the action at time $\tau_i(s+1)$ is chosen as $\sigma_i(\tau_i(s+1)) \in \arg\max_{\alpha_i\in A_i}U_i(\alpha_i,q_{-i}(\tau_i(s+1)-1))$.
Thus, by construction, the $(s+1)$-th action of player $i$ in the embedded (synchronous) FP process is chosen as $\tilde \sigma_i(s+1) \in BR_i(\hat q^i_{-i}(s)).$
In the embedded (synchronous) FP process, the term $\tilde q_j(s)$ may be thought of as the ``true'' empirical distribution of player $j$, and the term $\hat q_j^i(s)$ may be thought of as an estimate which player $i$ maintains of $\tilde q_j(s)$, and the term $\hat q^i(s)$ (note the superscript) may be thought of as player $i$'s estimate of the joint empirical distribution $\tilde q(s)$ at the time of player $i$'s $(s+1)$-th best response. Loosely speaking, if we can show that $\hat q^i(s)\rightarrow \tilde q(s),~\forall i$, then convergence of the embedded process $(\tilde q(s))$ (and eventually the original process $(q(n))$) will follow from the robustness result.

Before proceeding to the proof of Theorem \ref{thrm_async_FP}, we point out a few useful properties that will arise in the proof.
Note that for $i\in \mathcal{N}$ and $s\in \{1,2,\ldots\}$, we have
\begin{equation}
\label{ell_tau_eq}
N_i(\tau_i(s)) = s,
\end{equation}
\noindent
and for $i\in \mathcal{N}$ and $t\in \{1,2,\ldots\}$ we have
\begin{equation}
\label{X_i_implication}
X_i(n) = 1 \implies \tau_i(N_i(n)) = n.
\end{equation}
\noindent Furthermore, note that $X_i(n) = 0$ implies that $N_i(n) = N_i(n-1)$, and in particular,
\begin{align}
\label{q_step_equality}
X_i(n) = 0 \implies q_i(n) = q_i(n-1).
\end{align}
\noindent These facts are readily verified by conferring with the definitions of $\tau_i$, $N_i$, and $X_i$.

We now prove Theorem \ref{thrm_async_FP}.
\begin{proof}
As a first step, we wish to show that $\lim_{s\rightarrow\infty} d(\tilde q(s),~NE) = 0$. We accomplish this by invoking the robustness result. In particular, we wish to show that there exists a sequence $\{\epsilon_s\}_{s\geq 1}$ such that $\lim_{s\rightarrow\infty}\epsilon_s = 0$ and
\begin{equation} \label{thrm1_eq7}
U_i(\sigma_i(s+1),\tilde q_{-i}(s)) \geq \max_{y_i \in Y_i} U_i(\alpha_i,\tilde q_{-i}(s)) - \epsilon_s, ~\forall s\geq 1.
\end{equation}
To that end, for $i\in \mathbb{N}$ define $v_i:\Delta_{-i}\rightarrow \mathbb{R}$ by $v(q_{-i}) := \max_{y_i \in Y_i} U_i(\alpha_i,q_{-i})$, and note that by \eqref{eq_async_FP_proccess}, $U_i(\sigma_i(\tau_i(s+1)),q_{-i}(\tau_i(s+1)-1)) = v_i(q_{-i}(\tau_i(s+1)-1))$, or equivalently by the definitions of $\tilde \sigma(s)$ and $\hat q^i(s)$,
\begin{equation}
U_i(\tilde \sigma_i(s+1)),\hat q^i_{-i}(s)) = v_i(\hat q^i_{-i}(s)).
\label{thrm1_eq1}
\end{equation}
\noindent Using Lemma \ref{IR1} in the appendix,
it is straightforward to verify that $\lim_{s\rightarrow\infty} \|\hat q^i(s) - \tilde q(s)\| = 0$. Since $U_i$ is Lipschitz continuous, this gives
$\lim\limits_{s\rightarrow\infty} |U_i(\tilde \sigma_i(s+1)),\tilde q_{-i}(s)) - v_i(\tilde q_{-i}(s))| = 0, ~\forall i;$
i.e., there exists a sequence $\{\epsilon_s\}_{s\geq 1}$ such that $\epsilon_s \rightarrow 0$ and \eqref{thrm1_eq7} holds. It follows by Corollary \ref{cor_FP_robustness} that
\begin{equation}
\lim\limits_{s\rightarrow\infty} d(\tilde q(s),~NE) = 0.
\label{theorem_main_result_eq3}
\end{equation}

We now show that $\lim_{n\rightarrow\infty} d(q(n),~NE) =0$. Let $\varepsilon > 0$ be given. By Lemma \ref{IR1} (see appendix), for each $i\in \mathcal{N}$ there exists a time $S_i > 0$ such that $\forall s\geq S_i$, $\|q(\tau_i(s)) - \tilde q(s)\| < \frac{\varepsilon}{2}$. Let $S^{'} = \max_i\{S_i\}$. By \eqref{theorem_main_result_eq3} there exists a time $S^{''}$ such that $\forall s\geq S^{''}$, $d(\tilde q(s), ~NE) < \frac{\varepsilon}{2}$. Let $S=\max\{S^{'},S^{''}\}$. Then
\begin{equation}
d(q(\tau_i(s)),~NE) < \varepsilon, ~\forall i,~ \forall s\geq S.
\label{thrm1_eq6}
\end{equation}
\noindent Let $T = \max_{i}\{\tau_i(S)\}$. Note that for some $i$, $q(T) = q(\tau_i(S))$, and hence by \eqref{thrm1_eq6},
\begin{equation}
d(q(T),~NE) < \varepsilon.
\label{thrm1_eq4}
\end{equation}
\noindent Also note that for any $n_0>T$, it holds that $N_i(n_0) \geq S$ (since $N_i(\tau_i(S)) = S$, and $N_i(n)$ is non-decreasing in $n$), and moreover
\begin{align}
& X_i(n_0) = 1 \mbox{ for some } i ~\implies ~q(n_0) = q(\tau_i(N_i(n_0))),\\
& X_i(n_0) = 0 \mbox{ for all $i$ } ~\implies ~q(n_0) = q(n_0-1),
\label{thrm1_eq5}
\end{align}
\noindent where the first implication holds with $N_i(n_0) \geq S$. In the above, the first line follows from \eqref{X_i_implication}, and the second line follows from \eqref{q_step_equality}.
Consider $n\geq T$. If for some $i$, $X_i(n) = 1$, then by \eqref{thrm1_eq5} and \eqref{thrm1_eq6},
$d(q(n),~NE) = d(q(\tau_i(N_i(n))),~NE) < \varepsilon.$
Otherwise, if $X_i(n) = 0 ~\forall i$, then $q(n) = q(n-1)$.

Iterate this argument $m$ times until either (i) $X_i(n-m) = 1$ for some $i$, or (ii), $t-m = T$. In the case of (i),
$d(q(n),~NE) = d(q(n-m),~NE) = d(q(\tau_i(N_i(n-m))),~NE) < \varepsilon,$
where the inequality again
follows from \eqref{thrm1_eq6} and the fact that $n-m>T \implies N_i(n-m)\geq S$.
In the case of (ii),
$d(q(n),~NE) = d(q(T),~NE) < \varepsilon,$
where the inequality  follows from \eqref{thrm1_eq4}. Since $\varepsilon >0$ was arbitrarily, the result follows.
\end{proof}

\subsection{Continuous-Time Embedding of Fictitious Play} \label{sec_CT_implementation}
The asynchronous FP algorithm discussed in Section \ref{sec_async_FP} is a somewhat abstract discrete-time process. In this section we give a concrete interpretation of the process within a practical setting.
In particular, we consider the implementation of the (discrete-time) FP algorithm in a continuous-time setting where agents do not have access to a global clock. Effectively, this results in a discrete-time asynchronous FP process embedded within a continuous-time framework.

We first introduce the continuous-time embedding and derive a sufficient condition for convergence using Theorem \ref{thrm_async_FP}. Subsequently, we give two simple and practical implementations that achieve the condition.
The example implementations are prototypical in that one uses a synchronization rule that is entirely stochastic, and the other, entirely deterministic.

As in the the previous models of repeated play learning, assume each player executes a (countable) sequence of actions (or strategies) $\{\sigma_i(n)\}_{n\geq 1}$. Furthermore, assume that each action is taken at some instant in real time $t\in[0,\infty)$ as measured by some universal clock.\footnote{We use the term ``universal clock'' to refer to some reference clock by which we can compare the timing of actions taken by individual players. However, the universal clock is merely an artifice for analyzing the process, and we do not suppose that players have any particular knowledge concerning it.}
In particular, for each player $i$, let $\{\tau_i(n)\}_{n=1}^\infty \subset [0,\infty)$ be an increasing sequence where $\tau_i(n)$ indicates the time (as measured by the universal clock) at which player $i$ chooses an action for the $n$-th time. Let $\sigma_i(n)$ denote the $n$-th action taken by player $i$; i.e., the action taken by player $i$ at time $t=\tau_i(n)$. For $t\in[0,\infty)$, let $N_i(t)=\sup\{n:~\tau_i(n) \leq t\}$ denote the number of actions taken by player $i$ by time $t$. For $t\in[0,\infty)$, we define the empirical distribution of player $i$ in this settings as
$q_i(t) := \frac{1}{N_i(t)}\sum_{k=1}^{N_i(t)} \sigma_i(k).$
In particular, for $t\in[0,\infty)$, let $q_i(t_{-}) := \lim_{\tilde t \uparrow t} q_i(\tilde t)$.
%
%

In this context, we say the sequence $\{\sigma_i(n)\}_{n\geq 1}$ is an asynchronous FP action process if for $n\geq 1$ each player $i$ chooses their stage-$n$ action according to the rule:\footnote{Let $\tau_i(1) = 0$ for all $i$, and let the initial action $\sigma_i(1)$ be chosen arbitrarily for all $i$.}
\begin{align}
\sigma_i(n) \in BR_i(q_{-i}(\tau_i(n)_{-}))
\end{align}
We call the sequence $\{\tau_i(n)\}_{n\geq 1}$ the action-timing process for player $i$, and we refer to any method used to generate $\{\tau_i(n)\}_{n\geq 1}$ (whether deterministic or stochastic) as an action timing rule. Together, we refer to the joint sequence $\{\tau_i(n),\sigma_i(n)\}_{i\in\mathcal{N},n\geq 1}$ as a continuous-time embedded FP process.

The following assumption provides a sufficient condition on the action-timing process in order to ensure convergence of the continuous-time embedded FP process. The assumption is essentially a restatement of A. \ref{a_synchrony}, but in a continuous-time setting.
\begin{assumption} \label{a_synchrony_CT}
(i) For each $i$ there holds $\lim_{t\rightarrow\infty} N_i(t) = \infty$, (ii) for each $i,j$ there holds $\lim_{t\rightarrow\infty} N_i(t)/N_j(t)=1$.
\end{assumption}
Part (i) of the above assumption may be satisfied, for instance, as long as the clock skew of each agent stays bounded (with respect to the universal clock), and each agent takes actions infinitely often with respect to their local clock. In order to ensure (ii) is satisfied, slightly more care is needed, as demonstrated by the specific application scenarios below.

The following theorem demonstrates that if the action-timing sequence is chosen to satisfy A. \ref{a_synchrony_CT}, then the continuous-time embedding of FP will converge to the set of NE.
\begin{theorem} \label{thrm_CT_FP}
Let $\Gamma$ be a potential game. Suppose that $\{\sigma_i(n),\tau_i(n)\}_{i\in \mathcal{N},~n\geq 1}$ is a continuous-time embedding of FP satisfying A. \ref{a_synchrony_CT}. Then players learn NE strategies in the sense that $\lim{t\rightarrow\infty} d(q(t),NE)=0$.
\end{theorem}
The proof of Theorem \ref{thrm_CT_FP} follows readily from Theorem \ref{thrm_async_FP}.

In the following two subsections, we give two simple examples of action-timing rules that illustrate different methods for achieving A. \ref{a_synchrony_CT} (and hence achieving NE learning in the continuous-time embedded FP process).

\subsection{Independent Poisson Clocks} \label{sec_poisson_clocks}
Let $w_i(n) = \tau_i(n+1) - \tau_i(n)$ denote the stage $n$ ``waiting time'' for player $i$.
Suppose that for each player $i$ and $n\geq 1$, $w_i(n)$ is an independent random variable with distribution $w_i(n) \sim exp(\lambda)$, where $\lambda>0$ is some parameter that is common among all $i$. In this case, the action-timing process $\{\tau_i(n)\}_{n\geq 1}$ is said to be a homogenous Poisson process.

The following theorem shows that if the action-timing process is randomly generated in this manner, then players will achieve NE learning.
\begin{theorem}
Let $\Gamma$ be potential game. Suppose that players are engaged in a continuous-time embedded asynchronous FP process and the action-timing sequences $\{\tau_i(n)\}_{n\geq 1}$ are generated as independent homogenous Poisson processes with common parameter $\lambda$. Then players learn NE strategies in the sense that $\lim\limits_{t\rightarrow\infty} d(q(t),NE) = 0$, almost surely.
\end{theorem}
\begin{proof}
By Theorem \ref{thrm_async_FP} it is sufficient to show that $\lim_{t\rightarrow\infty} N_i(t) = \infty,~\forall i$, and $\lim_{t\rightarrow\infty}\frac{N_i(t)}{N_j(t)}=1$ for all $i,j$.

First, note that for any $i$ and $n\geq 1$, $w_i(n) < \infty$ almost surely. Hence, $\tau_i(n) = \sum_{k=1}^n w_i(k) <\infty$ for all $i$, almost surely. Equivalently, for any $M>0$, almost surely there exists a (random) time $T>0$ such that $N_i(t) \geq M$ for all $t\geq T$. Hence, $\lim_{t\rightarrow\infty} N_i(t) = \infty$, almost surely.

Now we show that $\lim_{t\rightarrow\infty}\frac{N_i(t)}{N_j(t)}=1$ for all $i,j$.
Let $\tau(1) := \min_i \tau_i(1)$ and let $\mathcal{T}_1 := \{\tau_i(n)\}_{i\in\mathcal{N},n\geq 1}\backslash \tau(1)$. For $n\geq 2$, let $\tau(n) := \min \mathcal{T}_{n-1}$ and let $\mathcal{T}_n := \mathcal{T}_{n-1}\backslash\tau(n)$. In this manner, we produce the sequence $\{\tau(n)\}$.
For $n\geq 1$, $i\in \mathcal{N}$, define $X_i(n) \in\{0,1\}$ to be an indicator variable with $X_i(n) = 1$ if $\tau(n) \in \{\tau_i(k)\}_{k\geq 1}$ and $X_i(n) = 0$ otherwise.

Let $\mathcal{F}_0 := \emptyset$ and for $n\geq 1$, let $\mathcal{F}_n := \sigma(\{\tau(k)\}_{k=1}^n )$. For $n\geq 1$ let $\xi_{i}(n) := \prob(X_i(n) = 1|~\mathcal{F}_{n-1})$.

Since for each $i$, $\{\tau_i(n)\}_{n\geq 1}$ is a Poisson process with common parameter $\lambda$, there holds $\xi_{i}(n) = \frac{1}{N}$ for all $i$ and $n$.\footnote{Recall that $N$ denotes the number of players.}
By Levi's extension of the Borel-Cantelli Lemma (see \cite{williams_book}, p.124) there holds
\begin{equation} \label{eq_Levi_BC}
\lim_{n\rightarrow\infty} (\sum_{k=1}^n X_i(n))/(\sum_{k=1}^n \xi_{i}(n)) = 1, \mbox{ a.s.}
\end{equation}
Note that for each $i$, $\sum_{k=1}^n X_i(k)=N_i(\tau(n))$ and $\sum_{k=1}^n \xi_i(n) = \frac{n}{N}$. Thus by \eqref{eq_Levi_BC},
$
\lim_{n\rightarrow\infty} \frac{N_i(\tau(n))}{N_j(\tau(n))} = \lim_{n\rightarrow\infty} \frac{N_i(\tau(n))}{n/N} \frac{n/N}{N_j(\tau(n))} = 1, \mbox{ a.s.},~\forall i,j.
$

Finally, note that $\lim_{n\rightarrow\infty}\tau(n) = \infty$ a.s., and for each $i$ $N_i(t)$ is constant on
\newline $[0,\infty)\backslash \{\tau(n)\}_{n\geq 1}$. Thus, $\lim_{t\rightarrow\infty} \frac{N_i(t)}{N_j(t)} = 1$, almost surely.
\end{proof}

\subsection{Adaptive Clock Rates} \label{sec_adaptive_clocks}
In this section we consider a scenario in which each player chooses the timing of her actions (deterministically) according to a personal clock with a skew rate that may be different among players.

Let $w_i(n) = \tau_i(n+1) - \tau_i(n)$ again denote the stage $n$ ``waiting time'' for player $i$. For each $i$, let $w_{i,0}$ denote a base waiting time for player $i$. The base waiting time of player $i$ may be interpreted as the amount of time which expires according to the universal clock during one unit of time as measured by player $i$'s personal clock. The disparity in the $w_{i,0}$ thus reflects disparate skew rates among players' personal clocks.

Let $N_{min}(t) := \min_{i} N_i(t)$. At time $t$, we suppose that player $i$ has knowledge of $N_{min}(s)$ at the time instances $s\in\{kw_{i,0}:~k\in \mathbb{N}_+, kw_{i,0}\leq t\}$. (I.e., player $i$ is aware of the value of $N_{min}$ at instances when her ``clock ticks''.) For each $i$, let $B_i\in\mathbb{R}$ be a number satisfying $B_i > \max_{i} w_{i,0}$.

Suppose that player $i$ adaptively chooses her stage $n$ waiting time according to the rule:
{\small
\begin{align} \label{eq_waiting_time}
w_i(n) =
\min \big\{kw_{i,0}:~ k\in \mathbb{N}_+, & ~ N_{min}(\tau_i(n) + kw_{i,0}) \geq N_i(\tau_i(n)) - B_i\big\}
\end{align}
}
In words, this rule may be described as follows: Player $i$ periodically observes $N_{min}(t)$. If $N_i(t) - N_{min}(t) \leq B_i$ then player $i$ takes a new action. If $N_i(t) - N_{min}(t) > B_i$ then player $i$ waits for $N_{min}(t)$ to increase sufficiently (satisfying $N_i(t) - N_{min}(t) \leq B_i$) before taking a new action.

\begin{theorem}
Let $\Gamma$ be a potential game. Suppose that players are engaged in a continuous-time embedded asynchronous FP process in which the action-timing sequence $\{\tau_i(n)\}_{n\geq 1}$ is generated according to the adaptive rule \eqref{eq_waiting_time}. Then players learn NE strategies in the sense that $\lim_{t\rightarrow\infty} d(q(t),NE)=0$.
\end{theorem}
\begin{proof}
By Theorem \ref{thrm_async_FP}, it is sufficient to show that $\lim_{t\rightarrow\infty} N_i(t)=\infty$ for some (and hence all) $i$, and that $\lim_{t\rightarrow\infty}\frac{N_i(t)}{N_j(t)}=1$.

Note that for $i^*\in \arg\max_{i} w_{i,0}$, there holds $N_{i^*}(t) = \lfloor\frac{t}{w_{i^*,0}} \rfloor + 1$, and hence $\lim_{t\rightarrow\infty}N_{i^*}(t) = \infty$. Furthermore, by construction, $|N_i(t) - N_{i^*}(t)| \leq 2\max_{i} B_i$ for all $i$ and for all $t\geq 0$. Hence, $\lim_{t\rightarrow\infty}\frac{N_i(t)}{N_j(t)}=1$, for all $i,j$.
\end{proof}

\section{Concluding Remarks} \label{sec_conclusion}
We have studied the robustness of a class of best-response based algorithms that we refer to as FP-type algorithms. It has been shown that the convergence of such algorithms can be retained under a form of best-response perturbation in which players are permitted to sometimes make errors in their best response action, so long as the degree of suboptimality asymptotically decays to zero. We have shown that this form of robustness can be used to develop practical algorithms, including distributed algorithms, reduced-complexity algorithms, and asynchronous algorithms.

\appendix

\section*{Appendix}
{\small
\begin{lemma}
Let $i,j\in N$, let $\tau_i(s)$ and $\tilde q_j(s)$ be defined as in Section
\ref{sec_async_FP}, and assume A. \ref{a_synchrony} holds. Then $\lim_{s\rightarrow\infty} \|  q_j(\tau_i(s)) - \tilde q_j(s)\| = 0.$
\label{IR1}
\end{lemma}
\begin{proof}
Note that by the definitions of $\tau_j$, $N_j$, and $\tilde q_j$ there holds $q_j(n) = q_j(\tau_j(N_j(n))) =  \tilde q_j(N_j(n)),$ for any $n\in \mathbb{N}_+$
Noting that $\sqrt{2}=\max_{p',p''\in \Delta(Y_j)} \|p'-p''\|$, we also have $\|\tilde q_j(s+1) - \tilde q_j(s)\| \leq \frac{\sqrt{2}}{s}$, for $s\in \mathbb{N}_{+}$, and more generally, for $s_1,s_2 \in \mathbb{N}_+$, we have
$ \|\tilde q_j(s_1) - \tilde q_j(s_2)\| \leq \sum_{s=\min(s_1,s2)}^{\max(s_1,s_2)-1} \|\tilde q_j(s+1) - \tilde q_j(s)\| \leq \frac{|s_2 - s_1|}{\min(s_1,s_2)}\sqrt{2}.$
Hence, $|q_j(\tau_i(s)) - \tilde q_j(s)\| = \|\tilde q_j(N_j(\tau_i(s))) - \tilde q_j(s)\| = \|\tilde q_j(N_j(\tau_i(s))) - \tilde q_j(N_i(\tau_i(s)))\| \leq \frac{\vert N_j(\tau_i(s)) - N_i(\tau_i(s))\vert}{\min(N_i(\tau_i(s)),N_j(\tau_i(s)))}\sqrt{2}, $
where the second equality follows from the fact that $N_i(\tau_i(s)) = s$ (see \eqref{ell_tau_eq}). Thus, it suffices to show that $\lim\limits_{s\rightarrow\infty} \frac{\vert N_j(\tau_i(s)) - N_i(\tau_i(s))\vert }{\min(N_i(\tau_i(s)),N_j(\tau_i(s)))} = 0.$
But, by A. \ref{a_synchrony}, for any $i,j$ there holds: $0=\lim_{n\rightarrow\infty}\frac{N_i(n)}{N_j(n)}-1=\lim_{s\rightarrow\infty} \frac{N_i(\tau_i(s))}{N_j(\tau_i(s))}-1 = \lim_{s\rightarrow\infty}\frac{N_i(\tau_i(s)) - N_j(\tau_i(s))}{N_j(\tau_i(s))} $, where the second equality follows from the fact that
(again by A. \ref{a_synchrony}) $\lim_{s\rightarrow\infty}\tau_i(s) = \infty$.
\end{proof}
}

\bibliographystyle{unsrt}
\bibliography{myRefs}

\end{document}